\documentclass[11pt]{amsart}

	\newtheorem{thm}{Theorem}
	
	\newtheorem{cor}[thm]{Corollary}
	\newtheorem{ex}[thm]{Example}
	
	\newtheorem{Def}[thm]{Definition}
		\newtheorem{prop}[thm]{Proposition}
		\newtheorem{rem}[thm]{Remark}
		\newcommand{\Rem}{\begin{rem} \rm}
			\newcommand{\bdfn}{\begin{Def} \rm}
				\newcommand{\edfn}{\end{Def}}

\begin{document}
				\large
				\title[   Quotient Lifting Properties]{ Order preserving Quotient Lifting Properties }
			
								\author[T. S. S. R. K. Rao]{T. S. S. R .K. Rao}
				
				\address{ Adjunct Professor, CARAMS, MAHE, Manipal, India  }
				\email{srin@fulbrightmail.org }

\maketitle
\vskip 1em

\begin{abstract}  Several  properties of subspaces $J \subset X$ in a pair $(X,J)$ with quotient lifting properties of different types are derived, giving emphasis to
order-preserving properties for order unit spaces and spaces of affine continuous functions on Choquet simplexes.   \end{abstract}
\vskip 1em
\section{Introduction}
The Quotient Lifting Property ($QLP$) for pairs of Banach spaces deals with the existence of norm preserving lifts of operators.  		
\begin{Def} ( \cite{mbf}) Let $X$ be a real Banach space and $J$ be a closed subspace of  $X$. The pair $(X,J)$ has the QLP if and only if for every Banach space $Y$ and every  bounded operator  $S: Y \to X/J$ there exists  a bounded operator $T$ from $Y$ to $X$ lifting $S$ while preserving the norm, i.e. $\|T\|=\|S\|$ and $\pi \circ T = S$, where $\pi: X \rightarrow X/J$ is the quotient map.
	\end{Def} 	
 If $(X,J)$ has the QLP then   $J$ is the kernel of a contractive projection and hence a  complemented subspace of $X$, see \cite{mbf}. 
\vskip 1em

 Motivated by this, we introduce a related and weaker property involving the lifting of non-negative bounded linear operators to have non-negative, norm-preserving lifting in the class of order unit spaces (see \cite{A} and \cite{H}, Chapter 1) and compact  operators to compact operators with the same norm. We also give conditions under which this property is preserved by subspaces and quotient spaces.
 \vskip 1em

 Let $K$ be a compact convex set in a locally convex topological vector space and let $A(K)$ denote the space of real-valued affine continuous functions on $K$, equipped with the supremum norm and point-wise order. For a closed split face $F \subset K$, if $J = \{a \in A(K): a(F)=0\}$, we show that for the pair $(A(K),J)$, the order preserving, linear quotient lifting property is equivalent to $F$ being the range of an affine continuous map on $K$, which is the identity on $F$.  We give examples of subspaces $J$ of  $A(K)$, for a Choquet simplex $K$, so that $(A(K),J)$ has the compact extension property , using vector-valued Choquet theory. Our approach follows the popular route of converting the lifting problem to a vector-valued extension problem.
 \vskip 1em
 
 The general question about how lifting properties behave with respect to injective tensor products, among various categories of Banach spaces, has not received much attention.  Unlike the QLP, for the compact quotient lifting property (CQLP), the subspace may not be complemented, making the passage to tensor products more difficult.  

\vskip 1em
Let $\Omega$ be a compact Hausdorff space and let $C(\Omega)$ denote the space of continuous functions on $\Omega$. We conclude the paper with an application of a remarkable result of Aviles and Troyanski relating the existence of uniformly rotund in every direction (URED) renorming of $C(\Omega)$ to that of a strictly positive measure supported on $\Omega$, to lifting lattice homomorphism preserving operators to $C(\Omega)$.

\section{Various forms of lifting properties}

The  results of this  section show an interdependence between  geometric properties of a Banach spaces and the lifting properties. We also give results showing  how these properties behave with respect to order ideals or lattice ideals, in specific function spaces.

\vskip 1em
Let $J \subset X$ be a closed subspace, we say that $(X,J)$ has the compact lifting property (CQLP), if for all Banach spaces $Z$ and compact linear operators,
$T: Z \rightarrow X/J$, there is a compact linear operator $S: Z \rightarrow X$ such that $\pi \circ S = T$ and $\|T\|=\|S\|$.
\vskip 1em

Clearly if $(X,J)$ has the QLP, since the lifting is via composition with the associated projection $P$, QLP implies the CQLP. More precisely, we recall from \cite{mbf} that $P$ is the projection defined by the composition of  $\pi : X \rightarrow X/J$ followed by $\tilde{Id}$, which is the norm-preserving lift
 of the identity map, $I: X/J \rightarrow X/J$. Also if $J$ is of finite codimension, then the CQLP implies the QLP.

\begin{ex}\label{ex?} Let $\Omega$ be a compact Hausdorff space and let $C(\Omega,X)$ be the space of $X$-valued continuous functions, equipped with the supremum norm.
We recall the identification of the space of compact operators ${\mathcal K}(X, C(\Omega))$ as the space  $C(\Omega,X^\ast)$ via the  embedding $T \rightarrow T^\ast \circ \delta$ from $\Omega $ into $X^\ast$ ( $\delta:\Omega \rightarrow C(\Omega)^\ast$ is the evaluation map) is an onto isometry.
For a closed set $E \subset \Omega$,
 if $J = \{f \in C(\Omega): f(E) = 0\}$, note that $C(\Omega)/J$ can be identified via the `restriction' map with $C(E)$,  and  hence
 $\mathcal{K}(X,C(\Omega)/J)$ is  identified with
 $C(E, X^\ast)$, an application of  Dugundji's extension theorem, implies that $(C(\Omega), J)$ has the CQLP.

\end{ex}

\vskip 1em
Given an infinite discrete set $\Gamma$, using the identifications of $\ell^\infty (\Gamma)$ with  $C(\beta(\Gamma))$ and  $c_0(\Gamma) $ with $ \{f \in C(\beta(\Gamma)): f(\beta(\Gamma)-\Gamma) = 0\}$, we have  that $(\ell^\infty(\Gamma), c_0(\Gamma))$ has the CQLP, but not the QLP as $c_0(\Gamma)$ is not even a complemented subspace of $\ell^\infty(\Gamma)$.

\vskip 1em
In the next proposition we give an application of the CQLP.

\begin{prop}
If  $(X,J)$ has the CQLP,  let $\{\pi(x_n)\}_{n \geq 1}$ be a relatively compact sequence in $X/J$. Then there is a relatively compact sequence $\{z_n\}_{n \geq 1} \subset X$ having the same bound as $\{\pi(x_n)\}_{n \geq 1}$ such that $\pi(z_n)= \pi(x_n)$ for all $n$.
\end{prop}
\begin{proof}
 It is easy to see that the canonical bounded linear map $T: \ell^1 \rightarrow X/J$ defined by $T(e_i) = \pi(x_i)$ (where $e_i$ is the canonical basis in $\ell^1$) is a compact operator. Now by hypothesis, we have $S: \ell^1 \rightarrow X$ a compact operator with $\pi \circ S = T$, $\|S\|=\|T\|$. It is easy to see that $\{z_i=S(e_i)\}_{i \geq 1}$ is the required sequence in $ X$.
\end{proof}
Our next two results deal with the stability of the property CQLP. In the following proposition we use the canonical embedding of a space in its bidual. In particular, $(X/J)^{\ast\ast} = X^{\ast\ast}/J^{\bot\bot}$ and the quotient map, $\pi: X^{\ast\ast} \rightarrow X^{\ast\ast}/J^{\bot\bot}$ is the biadjoint of the corresponding quotient map from $X \rightarrow X/J$, still denoted by $\pi$. If $J \subset X$ is a reflexive space, then $J^{\bot\bot} = J$. We recall that $J$ is said to be factor reflexive space, if $X/J$ is a reflexive space.
\begin{prop}
\begin{enumerate}
	\item Let $ J \subset Y \subset X$. Suppose $(X,J)$ has the CQLP. Then $(Y,J)$ has the CQLP.
	\item  If $J$ is factor reflexive and $(X,J)$ has the CQLP, then $(X^{\ast\ast},J^{\bot\bot})$ has the CQLP.
	\item
	If $J$ is reflexive, then $(X^{\ast\ast},J)$ has the CQLP implies $(X,J)$ has the CQLP.

\end{enumerate} 	
\end{prop}
\begin{proof}
	Let $T: Z \rightarrow Y/J \subset X/J$ be a compact operator. By hypothesis, there is a compact operator $S: Z \rightarrow X$ such that $\pi \circ S = T$ and $\|S\| = \|T\|$. 
	 For any $z \in Z$, $\pi(S(z))=T(z) = \pi(y)$, for some $y \in Y$. Hence $S(z) -y \in J$, so that $S(z) \in Y$. The conclusion follows.
	\vskip 1em
	Now suppose $J$ is factor reflexive and $(X,J)$ has the CQLP. Let $T: Z \rightarrow X^{\ast\ast}/J^{\bot\bot} = X/J$ be a compact operator. By hypothesis, there is a compact operator, $S: Z \rightarrow X$ such that $\pi \circ S = T$ and $\|S\|=\|T\|$. The conclusion in (2) follows in view of our remark about quotient maps.
	\vskip 1em
	When $J$ is reflexive, in view of the inclusions $J \subset X \subset X^{\ast\ast}$, the conclusion in (3) follows from (1).
\end{proof}
In the following proposition, we use the notation $\pi$ for several quotient spaces, with appropriate interpretation.

\begin{prop}
Let $Z \subset Y \subset X$	be Banach spaces. Suppose $Y$ is of finite codimension and $(X,Y)$ has the CQLP. Then $(X/Z,Y/Z)$ has the CQLP.
\end{prop}
\begin{proof}
Since $I: X/Y \rightarrow X/Y$	is a compact operator, by hypothesis, there is a compact operator $I': X/Y \rightarrow X$ such that $\pi \circ I'=I$ and $\|I'\|=1$. Now $P: X \rightarrow X$ defined by $P(x)=I'(\pi(x))$ is a contractive projection with $ker(P)=Y$. Define $P': X/Z \rightarrow X/Z$ by $P'(\pi(x))= \pi(P(x))$. If $\pi(x_1)=\pi(x_2)$, then $x_1-x_2 \in Z \subset Y$, so $P(x_1)=P(x_2)$. Thus $P'$ is well defined. It is routine to check that it is a contractive projection. If $y \in Y$, then $P'(\pi(y))= \pi(0)$, so $Y/Z \subset ker(P')$. Conversely if 
$P'(\pi(x))=0$, then $\pi(x)= \pi(x-P(x))$, as $(x-P(x)) \in Y$, we get $\pi(x) \in Y/Z$. Hence $ker(P')= Y/Z$. Now it follows from our remarks at the beginning of Section 2, that $(X/Z,Y/Z)$ has the CQLP.
\end{proof}

\vskip 1em 
	Let $X$ be an order unit space and let $J\subset X$ be a proper order ideal. Now for the order unit space $X/J$, we can consider the pair $(X,J)$ having the order preserving quotient lifting property (OQLP), if for any order unit space $Z$ and for any linear, bounded, order unit preserving operator $T: Z \rightarrow X/J$, there is a linear bounded, order unit preserving $S: Z \rightarrow X$ such that $\pi \circ S = T$ and $\|T\|=\|S\|$.
	\vskip 1em
	When the pair $(X,J)$ has an appropriate lattice structure, a similar lifting property can be defined. We will be implicitly using this without giving a formal definition. A version of Propositions 4 and 5 can be contemplated, when $X$ is also an ordered ideal in its bidual or $X/Y$ is an order unit space.
	\vskip 1em
	
In what follows we will be using the notations and terminology of convexity theory and order unit spaces from Chapters I and II of Alfsen's monograph. See also Section 18 of \cite{H}. We illustrate situations where apart from quotient lifting, one can also achieve `lifts' that preserve the order structure.

\vskip 1em

Let $K$ be a compact convex set in a locally convex topological vector space.  Let $F \subset K$ be a closed, extreme convex set (face) and let $\phi: K \rightarrow F$ be an affine continuous map which is the identity on $F$. It is easy to see that $P: A(K) \rightarrow A(K)$ defined by $P(a) = a \circ \phi$ is a positive, contractive linear projection such that $P(1) = 1$, $ker(P) = \{a \in A(K):a(F)=0\}$ and it is an order ideal in $A(K)$. It is also easy to see that the canonical map, $\Phi: A(K)/ker(P) \rightarrow A(K)$ defined by $\Phi(\pi(a))= P(a)$ is an order-preserving linear contraction. Now for any order unit space $E$, with order unit $e$, if $T: E \rightarrow A(K)/ker(P)$ is a linear, continuous, order-preserving map with $T(e) = \pi(1)$, then it is easy to see that $T' = \Phi \circ T $ is an order-preserving, norm-preserving quotient lifting.
\vskip 1em
In the following theorem we preserve the notation from above.

\begin{thm} Let $K$ be a compact convex set and let $F \subset K$ be a closed split face. Suppose for all order unit spaces $E$ and for all non-negative, bounded linear operators $T: E \rightarrow A(K)/J$, there is a non-negative bounded linear lifting $T': E \rightarrow A(K)$
which is also norm preserving. There exists an affine continuous map $\phi: K \rightarrow F$ such that $\phi$ is the identity on $F$. Hence $(A(K),J)$ has the QLP.

\end{thm}
\begin{proof}
	Since $F$ is a closed split face, we use the identification of $A(K)/J$ with $A(F)$ via the map, $\pi(a) \rightarrow a|F$ both as an isometry and order-preserving map. Now the identity map $I: A(K)/J \rightarrow A(K)/J$ is a linear order-preserving map and hence by hypothesis, admits a norm-preserving, order-preserving lifting $I': A(F) \rightarrow A(K)$.
	\vskip 1em
	We next recall that $K$ can be identified with the state space $\{\Lambda \in A(K)^\ast: \|\Lambda\|=\Lambda(1)=1\}$, equipped with the weak$^\ast$-topology,  via the affine homeomorphism $k \rightarrow \delta(k)$ where $\delta(k)$ is the evaluation map.
	Now since $A(K)/J=A(F)$ is an order unit space, applying the hypothesis to $I: A(F) \rightarrow A(K)/J$, let $I': A(F) \rightarrow A(K)$ be the order unit , norm preserving quotient lift. It is easy to see that using implicitly the canonical embeddings noted above, $(I')^\ast|K \rightarrow F$ is the required continuous, affine map. The quotient lifting property implies that this is the identity on $F$.
\end{proof}
We next prove a vector-valued  affine continuous function version of Example 2. This requires some knowledge of Choquet simplexes from Alfsen's monograph
and a vector-valued integral representation analogue
of Singer's theorem (see \cite{DU}), due to Saab \cite{Sa}. See also \cite{DU} for general theory of vector measures and tensor product theory. For a Choquet simplex $K$, we first establish the canonical isometry between $A(K,X^\ast)$, the injective tensor product space $A(K) \otimes_{\epsilon}X^\ast$ and ${\mathcal K}(X,A(K))$.

\vskip 1em
Let $K$ be a compact Choquet simplex and let $F \subset K$ be a metrizable closed face. Since $K$ is a simplex, $F$ is a split face and  also a Choquet simplex (see Chapters II and III of \cite{A}). Note that the set of extreme points $\partial_e F = \partial_e K \cap F$. Since $F$ is metrizable, $\partial_e F$ is a $G_{\delta}$-subset of $F$ and as $F$ is closed, $\partial_e F$, is a Borel set in $K$.

\vskip 1em
Since $K$ is a simplex, $A(K)$ has the metric approximation property (as $A(K)^\ast$ is isometric to a $L^1(\mu)$ space, see \cite{H} Section 18) , consequently for any Banach space $X$,  ${\mathcal K}(X,A(K))$ can be canonically identified with the injective tensor product space, $A(K) \otimes_{\epsilon}X^\ast$.
\vskip 1em
Let $\delta: K \rightarrow A(K)^\ast$ be the canonical evaluation map, which is an affine homeomorphism onto $\delta(K)$ when the latter is equipped with the weak$^\ast$-topology. Now as in the case of vector-valued continuous functions,
for $T \in {\mathcal K}(X,A(K))$, $ T^\ast \circ \delta : K \rightarrow X^\ast \in A(K,X^\ast)$ and conversely for any $a \in A(K,X^\ast)$, $T: X \rightarrow A(K)$ defined by $T(x) = x \circ a$ for $x \in X$ is a compact operator. It is easy to see that this association is an onto isometry between ${\mathcal K}(X,A(K))$ and $A(K,X^\ast)$. In particular we have identified
$A(K)\otimes_{\epsilon} X^\ast$ and $A(K,X^\ast)$. 
\vskip 1em
In what follows these identifications will be used implicitly. For a Banach space $X$ we recall that $X^\ast_1$ denotes the closed unit ball with the weak$^\ast$-topology. We recall that for $k \in \partial_e K$ and $e^{\ast\ast} \in \partial_e X^{\ast\ast}_1$, $\delta(k)e^{\ast\ast}$ denotes the extreme functional or boundary measure (in the sense in \cite{Sa}), $(\delta(k)e^{\ast\ast} )(a) = e^{\ast\ast}(a(k))$ for $a \in A(K,X^\ast)$.
\begin{thm} Let $K$ be a Choquet simplex and let $F \subset K$ be a closed metrizable face. Let $J = \{a \in A(K): a(F)=0\}$. Then $(A(K),J)$ has the CQLP.
\end{thm}
\begin{proof}
Let $X$ be any Banach space. We first note that since $A(K)/J=A(F)$ and $F$ is a simplex, ${\mathcal K}(X,A(F))$ has the canonical identifications as $A(F,X^\ast)$ and $A(F) \otimes_{\epsilon}X^\ast$. Our theorem will be proved once we show that any $a \in A(F,X^\ast)$ has a norm-preserving extension
in $A(K,X^\ast)$.

\vskip1em
To do this we show that if $J' = \{a \in A(K,X^\ast): a(F)=0\}$, then the  map $\Phi: A(K,X^\ast)/J' \rightarrow A(F,X^\ast)$ defined by $\Phi(\pi(a))=a|F$ is a surjective isometry and for any $b \in A(F,X^\ast)$, if $\Phi(\pi(a))=b$, then there is a $c \in J'$ such that $\|a-c\|=d(a,J) = \|a|F\|$
(i.e., $J'$ is a proximinal subspace of $A(K,X^\ast)$). Thus $a-c$ is a desired norm-preserving extension.

\vskip 1em
It is easy to see that $\Phi$ is a one-one, linear contraction. Fix $\pi(a)$. Let $\Lambda \in \partial_e (J')^\bot_1$ be such that $\Lambda(a)= d(a,J')$. We next use the metrizability of $F$ and vector-valued Choquet theorem of Saab, to identify $J'$ as the space of $X^{\ast\ast}$-valued boundary measures supported on $\partial_e F$. Once this is done, $\Lambda = \delta(k)x^{\ast\ast}$ for some $k \in F$ and $x^{\ast\ast} \in \partial_e X^{\ast\ast}_1$. Thus $\Lambda(a)= d(a,J')= x^{\ast\ast}(a(k)) \leq \|a(k)\| \leq \|a|F\|$. Therefore $\Phi$ is an into isometry.

\vskip 1em
Let $\tau \in (J')^\bot$ be a boundary measure. Fix $x^\ast \in X^\ast$, for any $a \in A(K)$, $a(F)=0$, since $ax^\ast \in J$, for the scalar boundary measure $x^\ast \circ \tau$, $\int_K a d(x^\ast\circ \tau) = 0$. Now by the scalar theory, by the metrizability of $F$,  $x^\ast \circ \tau$ is supported on $\partial_e F$. Therefore, $\tau$ is supported on $F$. The other inclusion is easy to see.

\vskip 1em
Now let $\lambda$ be a boundary measure in $A(F,X^\ast)^\ast_1$ vanishing on the range of $\Phi$. Since $F$ is metrizable, we again have, $\lambda$ is supported on $\partial_e F$. Now extend $\lambda$ to a regular $X^{\ast\ast}$-valued Borel measure on $K$. Since $\partial_e F$ is a Borel set, we see that $\lambda$ is a boundary measure in $A(K,X^\ast)^\ast_1$. Also for any $a \in A(K)$, $x^\ast \in X$, $\int_{\partial_e F} a d(x^\ast \circ \lambda)=0$. Since $K$ is a simplex, we again have $\lambda =0$. Thus $\Phi$ is onto.
\vskip 1em
Let $M(F^c,X^{\ast\ast})$ denote the space of $X^{\ast\ast}$-valued boundary measures on the Borel $\sigma$-field, supported by the open convex set $F^c$.
As in the scalar-valued case, by the integral representation theorem, it is easy to see that 
$A(K,X^\ast)^\ast = (J')^\bot \bigoplus_1 M(F^c,X^{\ast\ast})$
is an $\ell^1$-direct sum. Therefore $J'$ is a $M$-ideal in the sense defined in Chapter I of \cite{hww} and hence is a proximinal subspace of $A(K,X)$ by Proposition II.1.1 of \cite{hww}. (See also \cite{R} for a simpler proof of this proposition). This completes the proof.

\end{proof}

Our next corollary is an application of the proof of the above theorem  to $A(K,X)$, to describe the state space (always considered with the weak$^\ast$-topology) of certain vectors in $A(K,X)$.
\begin{cor} Let $K$, $F$  be as in the above theorem
and let $J'=\{a \in A(K,X): a(F)=0\}$. Let $0 \neq a \in A(K,E)$ be such that $d(a,J') < \|a\|$. Then the state space, $S_a = \{ \Lambda \in A(K,X)^\ast_1: \Lambda(a) = \|a\|\}= \overline{CO}\{\delta(k)e^\ast: (\delta(k)e^\ast)(a)=e^\ast(a(k))= \|a\|~,~k \in \partial_e F^c~,e^\ast \in \partial_e X^\ast_1\}$. If $\tau \in \partial_e S_a$, then $\tau = \delta(k)e^\ast$ for some $e^\ast \in \partial_e E^\ast_1$ and an accumulation point $k$ of $F^c$.
	
\end{cor}
\begin{proof}
Clearly the set on the R. H. S, when non-empty, is contained in $S_a$. Since $S_a$ is a face, if $\lambda \in \partial_e S_a$, then $\lambda \in \partial_e A(K,X^\ast)_1$. In view of the $\ell^1$-decomposition, similar to the one, given at the end of the above Theorem, applied now to $A(K,X)^\ast$,  as $d(a,J) < \|a\|$, we get $\lambda \in \partial_e M(F^c,X^\ast)_1$. As before, it is easy to see that $\lambda = \delta(k)e^\ast$ for some $k\notin F$   and $e^\ast \in \partial_e X^\ast_1$ (in particular, the set in the R. H. S in the formula is non-empty). Thus by an application of the Krein-Milman theorem, the first conclusion follows.
\vskip 1em
If $\tau \in \partial_e S_a$, then by Milman's converse of the Krein-Milman theorem, and using the compactness of $K$ and $X^\ast_1$,  there are nets $\{k_{\alpha}\}_{\alpha \in \Delta} \subset F^c$ and $\{e^\ast_{\alpha}\}_{\alpha \in \Delta} \subset \partial_e X^\ast_1$, $\delta(k_{\alpha})e^\ast_{\alpha} \rightarrow \tau$ in the weak$^\ast$-topology,  which we may and do assume  are also such that $k_{\alpha} \rightarrow k$ and $e^\ast_{\alpha} \rightarrow e^\ast$, for some $k \in K$ and $e^\ast \in X^\ast_1$. 
\vskip 1em
Thus $\tau = \delta(k)e^\ast$. As $\|\tau\| = 1$, we have, $\|e^\ast\|=1$. Since $\tau \in \partial_e A(K,X)^\ast_1$, it is easy to see that $e^\ast \in \partial_e X^\ast_1$.
\end{proof}
We do not know how to prove compact and order-preserving lifting theorems in the category of 
order unit spaces or lattices $X$, and $A(K,X^\ast)$ and for subspaces of the type $J$ as above.
\vskip 1em

	A particularly interesting situation occurs by taking $X = A(K_2)$ for a  Choquet simplex $K_2$. In the following theorem we again get a vector-valued, this time non-negative extension theorem. We only present an outline proof of this result, omitting the technical details.
	\begin{thm}
		Let $K_1,K_2$ be simplexes . Let $F \subset K_1$ be a closed face. Let $f \in A(F,A(K_2))$ be a non-negative, affine continuous function. $f$ has a non-negative, norm-preserving extension in $A(K_1,A(K_2))$.
	\end{thm}
	
	\begin{proof}
	It can be shown using the analysis carried out above and the tensor product theory of simplexes developed by I. Namioka and R. R. Phelps, in \cite{NP}, that $A(K_1,A(K_2))$ is isometrically and in order preserving way identified with $A(K')$ for a unique (upto affine homeomorphisms) simplex $K'$( formally called the tensor product of $K_1,K_2$). Now if $F$ is a closed face of $K_1$ then there is a closed face $G$ of $K'$, which corresponds to the tensor product of the simplexes $F,K_2$. Thus the vector-valued extension theorem is reduced to extending a non-negative affine continuous function on $G$ to a non-negative affine continuous function on $K'$ with the same norm, which in the scalar case was done in Chapter 3 of \cite{A}. 
\end{proof}
We next give another application of lifting lattice homomorphisms to a geometric notion.
\vskip 1em

Given  a compact set $\Omega$, a positive regular Borel measure $\mu$ is called strictly positive, if $\mu(U)>0$,
for any non-empty open set $U\subset \Omega$. For example, if $\Omega$ is also separable, with a dense set $\{\omega_n\}_{n \geq 1}$, then $\mu =\sum _1^\infty \frac{1}{2^n}\delta(\omega_n))$ is a strictly positive regular  measure. As before, we denote by $\delta(w)$ the point evaluation functional at $w$ or the Dirac measure concentrated in $w$.
\vskip 1em
We also recall the definition of uniform rotundity  in every direction (URED).
\begin{Def}\label{URED}
A norm on a space $X$  is said to be URED if whenever $\{x_n\}_{n\geq 1}$ and $\{y_n\}_{n\geq 1}$ are two sequences of unit vectors satisfying the the following two conditions:\\
\begin{itemize}
\item $\lim_n \|\frac{x_n+y_n}{2}\| =1;$
\item there exists $z \in X$  and scalar $r_n$ such that $x_n-y_n=r_n z,$ for all $n \in \mathbb{N},$
\end{itemize}
then $lim_n\|x_n-y_n\|=0.$
\end{Def}
If the second hypotheses on the definition is removed then we have the normal uniform rotundity.
The main theorem from \cite{at} implies that $C(\Omega)$ admitting an equivalent URED norm is equivalent to $\Omega$ to supporting a strictly positive measure.

\begin{thm}
Let $\Omega$ be a compact space such that $C(\Omega)$ admits an equivalent URED norm. Let $E \subset \Omega$ be a closed set and let $J = \{f \in C(\Omega): f(E)=0\}$. Suppose for every Banach lattice $X$, every bounded linear lattice homomorphism, $T: X \rightarrow C(\Omega)/J$ has a norm-preserving linear lifting $T'$, which is also a lattice homomorphism.	 Then $C(\Omega)/J$ has an equivalent URED norm.
\end{thm}
\begin{proof}
As before we use the canonical identification of the space $C(\Omega)/J$ as $C(E)$. Hence by hypothesis, the identity map $I: C(E) \rightarrow C(\Omega)/J=C(E)$ admits a linear norm-preserving lifting $I': C(E) \rightarrow C(\Omega)$ which is also a lattice homomorphism. Now by the structure of scalar-valued lattice homomorphisms on continuous function spaces, if $\phi: \Omega \rightarrow C(\Omega)^\ast_1$ is the canonical embedding, then $(I')^\ast \circ \phi: \Omega \rightarrow E$ is a continuous map which is the identity on $E$.

\vskip 1em
Now by the main result of \cite{at}, the existence of an equivalent URED norm implies that there is a strictly positive measure $\mu$ supported on $\Omega$. Now $\nu = \mu \circ\phi^{-1} $ is a strictly positive measure supported on $E$. Therefore by the theorem from \cite{at} again $C(E)=C(\Omega)/J$ admits an equivalent URED norm.
\end{proof}
\begin{rem}
As before the lattice lifting assumption implies the existence of a linear lattice, contractive projection $P: C(\Omega) \rightarrow C(\Omega)$ such that $ker(P) = J$. 
\end{rem}

\vskip 1em

	The author thanks Professors F. Botelho and R. J. Fleming for providing the initial impetus for this work. Thanks are also due to the Editorial team for the prompt handling of this article during the pandemic. The author's first article on Choquet simplexes appeared in 1982 and over four decades they continue to provide nice illustrations for Banach space theoretic concepts.

\end{document}